\documentclass[12pt,reqno] {amsart}
\usepackage{amsfonts,amssymb,amsmath,amsthm}
\usepackage{mathrsfs}
\usepackage{url}
\usepackage{enumerate}
\usepackage{xcolor}
\usepackage{todonotes}
\usepackage{ulem}

\newtheorem{theorem}{Theorem}[section]
\newtheorem*{theorem*}{Theorem}
\newtheorem{lemma}[theorem]{Lemma}

\newtheorem{prop}[theorem]{Proposition}

\theoremstyle{definition}

\newtheorem{rem}[theorem]{Remark}

\author[M. Bhattacharjee]{Monojit Bhattacharjee}
\author[R. Gupta]{Rajeev Gupta}
\author[V. Venugopal]{Vidhya Venugopal}

\address[M. Bhattacharjee]{Department of Mathematics, Birla Institute of Technology and Science, Pilani, K K Birla Goa Campus, Zuarinagar, Sancoale, Goa 403726, India
}

\address[R. Gupta]{School of Mathematics and Computer Science\\
Indian Institute of Technology Goa, India}

\address[V. Venugopal]{Department of Mathematics \\
Birla Institute of Technology and Science K K Birla Goa Campus, India
 }

\begin{document}
\title{Wold-type decomposition for Doubly commuting two-isometries}

\begin{abstract}
In this article, we prove that any pair of doubly commuting $2$-isometries on a Hilbert space has a Wold-type decomposition. 
Moreover,  
the analytic part of the pair is unitary equivalent to the pair of multiplication by coordinate function on a Dirichlet-type space on the bidisc. 
\end{abstract}

\maketitle

\section{Introduction}
Let $\mathcal{H}$ be a Hilbert space and $\mathcal{B}(\mathcal{H})$ denote the set of all bounded linear operators on $\mathcal{H}$. An operator $T \in \mathcal{B}(\mathcal{H})$ is said to be \textit{concave} if for every $h\in\mathcal H,$
\begin{eqnarray}\label{concave}
    \|T^2h\|^2 - 2 \|Th\|^2  + \|h\|^2  \leq 0.
\end{eqnarray}
An operator $T\in\mathcal B(\mathcal H)$ is called
 a \textit{$2$-isometry} if for each $h\in\mathcal H,$ \eqref{concave} holds with equality. 
 A closed subspace $\mathcal{W} \subseteq \mathcal{H}$ is said to be a \textit{wandering subspace} for $T$ if $ \mathcal{W}  \perp T^n \mathcal{W}$ for all $n \geq 1$ and an operator $T \in \mathcal{B}(\mathcal{H})$ is said to have \textit{wandering subspace property}  if $\mathcal{H} = \bigvee_{n \geq 0} T^n \mathcal{W}$ for some wandering subspace $\mathcal{W}$ of $T$. For an operator \( T \in \mathcal{B}(\mathcal{H}) \) with the wandering subspace property and a wandering subspace \( \mathcal{W} \), it holds that \( \mathcal{H} \ominus T\mathcal{H} = \mathcal{W} \), which shows that the wandering subspace is uniquely determined. Throughout the article, the subspace $ \mathcal{H} \ominus T\mathcal{H}$ shall be denoted by $\mathcal E_T.$

A key and highly useful result in the study of bounded linear operators on a Hilbert space is the Wold decomposition theorem. The classical version states that any isometry \( V \in \mathcal{B}(\mathcal{H}) \) (i.e., \( V^*V = I_{\mathcal{H}} \)) is a direct sum of a unitary and an unilateral shift (see \cite{W} and also see page 3 in \cite{NF}). 
The Wold decomposition and results analogous to it have profound impact in Operator theory, Operator algebra and also invariant subspace problem for Hilbert space of holomorphic functions (see \cite{KM, Ma, KMN,  SZ}). 
Wold-type decompositions of $2$-isometric operators also have been studied by many authors, see \cite{Richter1, Shimorin, Olofsson}. 
For a $2$-isometry $T \in \mathcal{B}(\mathcal{H})$, it is well-known that the subspace $\displaystyle \cap_{n \geq 0} T^n \mathcal{H}$ (:= $\mathcal{H}_0$) is $T$-reducing and $T|_{\mathcal{H}_0}$ is unitary (see \cite[Proposition 1.1]{Olofsson}). The operator $T \in \mathcal{B}(\mathcal{H})$ is said to be \textit{analytic} provided $\mathcal{H}_0 = \{0\}$. 
In \cite{Richter2}, the author proves the following Wold-type decomposition for a the class of analytic concave operators.
\begin{theorem}[\cite{Richter2}]\label{analytic+2-isometry}
    Let $T \in \mathcal{B}(\mathcal{H})$ be an analytic concave operator.
    Then $\mathcal{H} \ominus T\mathcal{H}$ is a wandering subspace for $T.$
\end{theorem}
Later, Shimorin extended this result for general bounded left-invertible operators (see \cite[Theorem 3.6]{Shimorin}). For the sake of completeness of this article, we state this result only for the class of $2$-isometries in Theorem \ref{olofsson-analytic}. 
\begin{theorem}[\cite{Olofsson}]\label{olofsson-analytic}
    Let $T$ be a $2$-isometry on $\mathcal{H}$. Then there exists a positive $\mathcal{B}(\mathcal{E}_T)$-valued operator measure $\mu$  on the unit circle $\mathbb{T}$ 
    such that 
    \[
    \big(T, \mathcal{H} \big) \cong \left( \begin{bmatrix} 
      U &  0 \\
      0                  & M_z
    \end{bmatrix} , \mathcal{H}_0 \oplus \mathcal{D}_{\mathcal E_T}(\mu) \right),
    \] where $\mathcal{H}_0$ is $T$-reducing subspace, $U$ is unitary and $M_z$ on $\mathcal D_{\mathcal E_T}(\mu)$ is an analytic $2$-isometry.   
\end{theorem}
For the definition and properties of $\mathcal D_{\mathcal E_T}(\mu)$ space, we refer the reader to \cite{Richter1, Olofsson, Omar}.

A natural question which has been an interest to many operator theorist is that `Can an analogous result be obtained in multivariate setting?'
There are many interesting results for commuting isometries as well as for the class of pairs of commuting analytic $2$-isometries with some additional assumptions on the pair (see \cite{Slocinski, Ma, JS,  ABJS, Chavan-Reza, MRV} and references therein). It is presumable that some extra assumptions on pair of commuting $2$-isometries are necessary to achieve results analogous to Theorem \ref{analytic+2-isometry} and Theorem \ref{olofsson-analytic}. 
In this article, we shall consider the class of pairs of doubly commuting $2$-isometries. Here is the definition of doubly commuting tuple of operators.
An $n$-tuple of commuting bounded linear operators $T= (T_1,\ldots,T_n)$ acting on a Hilbert space $\mathcal{H}$ is said to be \textit{doubly commuting} if $ \displaystyle T_i^*T_j = T_jT_i^*$ for any $1 \leq i \neq j \leq n$. The natural examples of doubly commuting tuples are those of the co-ordinate wise multiplication operators 
on the Hardy space, the Bergman space, and the Dirichlet space over the bidisc $\mathbb{D}^2$ (defined in Section \ref{Preliminaries} and also see \cite{MRV}).  

A Wold-type decomposition for a pair of doubly commuting isometries is derived in \cite{Slocinski} (see also \cite{CPS, KO, Popovici, SSW}).
More precisely,
\begin{theorem}(M. Slocinski)\label{Slocinski} Let $V = (V_1, V_2)$ be a pair of doubly commuting isometries on
a Hilbert space $\mathcal{H}$. Then there exists a unique decomposition
\[ 
\mathcal{H} = \mathcal{H}_{ss} \oplus \mathcal{H}_{su} \oplus \mathcal{H}_{us} \oplus  \mathcal{H}_{uu},
\] where $\mathcal{H}_{ij}$ is joint $V$-reducing subspace of $\mathcal{H}$ for all $i, j = s, u.$ Moreover, $V_1$ on $\mathcal{H}_{ij}$ is a
shift if $i = s$ and unitary if $i = u$ and that $V_2$ is a shift if $j = s$ and unitary if $j = u$.
\end{theorem}

In Theorem \ref{wold-type decomposition}, we establish a Wold-type decomposition for any pair of doubly commuting $2$-isometries acting on a Hilbert space as a generalization of Theorem \ref{olofsson-analytic}. The essence of analytic structure in it also gives a more finer version of Theorem 5.1 in \cite{NZ}.  
The novelty of this work lies in the treatment of a general pair of doubly commuting $2$-isometries, which are not necessarily cyclic or analytic, and their connection to shifts on Dirichlet spaces. As a corollary of this theorem one can derive Theorem \ref{Slocinski}.   


Before we turn our attention to Theorem \ref{wold-type decomposition}, in Section \ref{Preliminaries} we note a few notations and preliminaries.

\section{Preliminaries} \label{Preliminaries}
    Let $\mu_1$ and $\mu_2$ to be two positive $\mathcal{B}(\mathcal{E})$-valued operator measures defined on the unit circle $\mathbb{T}$, where $\mathcal{E}$ is a Hilbert space. For $i=1,2$, the Poisson integral of $\mu_i$ is denoted by $P[\mu_i]$ and is given by  
    \[
    P[\mu_i](z_i) =  \int_{\mathbb{T}} P(z_i,e^{i\theta}) d\mu_i(e^{i\theta})
    \] where $P(z_i,e^{i\theta}) = \frac{(1-|z_i|^2)}{|e^{i\theta}-z_i|^2}$ for any $z_i \in \mathbb{D}$ and $\theta\in\mathbb R.$ Note that $P[\mu_i](z_i)$ is a positive operator in $\mathcal{B}({\mathcal{E}}).$  Let the space of all $\mathcal{E}$-valued analytic functions on the bidisc be denoted by $\mathcal{O}(\mathbb{D}^2, \mathcal{E})$. Now, for any $ f \in \mathcal{O}(\mathbb{D}^2, \mathcal{E})$, we define the terms $D^{(2)}_{\mu_1,\mu_2, i}(f)$ for $i=1,2,3,$  
\begin{align*}
     D^{(2)}_{\mu_1,\mu_2,1}(f) &:= \lim_{r \rightarrow 1^-}  \iint_{\mathbb{D}\times \mathbb{T}} \left\langle P[\mu_1](z_1) \frac{\partial f}{\partial z_1}(z_1,re^{it}), \frac{\partial f}{\partial z_1}(z_1,re^{it}) \right\rangle_{\mathcal{E}} \,\,dt ~~ dA(z_1) \\
     D^{(2)}_{\mu_1,\mu_2,2}(f) &:= \lim_{r \rightarrow 1^-}  \iint_{\mathbb{T} \times \mathbb{D}} \left\langle P[\mu_2](z_2) \frac{\partial f}{\partial z_2}(re^{it},z_2), \frac{\partial f}{\partial z_2}(re^{it},z_2) \right\rangle_{\mathcal{E}} \,\, dA(z_2) ~~dt \\
     D^{(2)}_{\mu_1,\mu_2,3}(f) &:=   \iint_{\mathbb{\mathbb{D}} \times \mathbb{D}} \left\langle P[\mu_1](z_1)  P[\mu_2](z_2) \frac{\partial^2 f}{\partial z_1\partial z_2}(z_1,z_2), \frac{\partial^2 f}{\partial z_1 \partial z_2}(z_1,z_2) \right\rangle_{\mathcal{E}} \,\, dA(z_1) ~~dA(z_2),
\end{align*}
    where $dt$ denotes the normalized arc length measure  and $dA$ denotes the normalized area measure. 
    Furthermore, we define the Dirichlet-type integral of a function $f$ on bidisc as 
    \[ D^{(2)}_{\mu_1,\mu_2}(f):= D^{(2)}_{\mu_1,\mu_2,1}(f) + D^{(2)}_{\mu_1,\mu_2,2}(f)+D^{(2)}_{\mu_1,\mu_2,3}(f).
    \]
    Combining all these, the \textit{Dirichlet-type space over bidisc for the operator-valued measures ${\mu_1,\mu_2}$} on $\mathbb{T}$, denoted by $\mathcal{D}^2_{\mathcal{E}}({\mu_1,\mu_2})$, is defined to be 
    the space of all $\mathcal{E}$-valued analytic functions $f$ such that $D^{(2)}_{\mu_1,\mu_2}(f)$ is finite.
    The norm of $f \in \mathcal{D}^2_{\mathcal{E}}(\mu_1,\mu_2)$ is denoted by $\|f\|_{\mu_1,\mu_2}$ and is given by 
\begin{equation}\label{norm of D-mu}
    \|f\|_{\mu_1,\mu_2}^2 := \|f\|_{H^2(\mathbb{D}^2)}^2 + D^{(2)}_{\mu_1,\mu_2,1}(f) +  D^{(2)}_{\mu_1,\mu_2,2}(f) +  D^{(2)}_{\mu_1,\mu_2,3}(f),
\end{equation}
   The space $\mathcal{D}^2_{\mathcal{E}}(\mu_1,\mu_2)$ turns out to be a Hilbert space with respect to the norm given in \eqref{norm of D-mu}. In \cite{MRV}, the scalar-valued Dirichlet-type space (that is, when $\mathcal{E} = \mathbb{C}$) is studied in some detail. 
    
    Let $\mathbb Z_{\geq 0}$ denote the set of all non-negative integers. Throughout the article, for $m,p \in \mathbb{Z}_{\geq 0},$ we adopt  $(m \wedge p)$ as a concise notation for  $\min\{m,p\}.$
   For any $f \in  \mathcal{D}^2_{\mathcal{E}}({\mu_1,\mu_2})$ with $f(z_1,z_2) = \sum_{m,n =0}^\infty a_{m,n} z_1^m z_2^n$, the following identities can easily be derived (see \cite{MRV} for scalar-valued case) 
\begin{align*}
    \|f\|_{H^2(\mathbb{D}^2)}^2 = \sum_{m,n=0}^\infty \|a_{m,n}\|^2
\end{align*} 
\begin{equation}\label{dirichlet integral_1}
    D^{(2)}_{\mu_1,\mu_2,1}(f) = \lim_{r \rightarrow 1^-} \sum_{m,p=1}^\infty \sum_{n=0}^\infty (m \wedge p) \langle \hat{\mu_1}(p-m) a_{m,n}, a_{p,n}\rangle r^{2n} 
\end{equation}
\begin{equation}\label{dirichlet integral_2}
     D^{(2)}_{\mu_1,\mu_2,2}(f) = \lim_{r \rightarrow 1^-}  \sum_{m=0}^\infty \sum_{n,q=1}^\infty (n \wedge q) \langle \hat{\mu_2}(n-q) a_{m,n}, a_{m,q}\rangle r^{2m} 
\end{equation}
\begin{equation}\label{dirichlet integral_3}
     D^{(2)}_{\mu_1,\mu_2,3}(f) =  \sum_{m,n=1}^\infty \sum_{p,q=1}^\infty (m \wedge p) (n \wedge q) \langle \hat{\mu_2}(q-n)\hat{\mu_1}(p-m) a_{m,n}, a_{p,n}\rangle, 
\end{equation} 
where $n^{th}$-Fourier coefficient of the measure $\mu$ is defined by 
\[\hat{\mu}(n):=\int_{\mathbb T}\overline{x}^n d\mu(x).\]
     Let $M_z=(M_{z_1},M_{z_2}),$ where each $M_{z_i}$ corresponds to multiplication by the co-ordinate function $z_i$ for $i=1,2.$
     Note that, the operators $M_{z_1}$ and $M_{z_2}$ are bounded linear operators on $\mathcal{D}^2_{\mathcal{E}}(\mu_1,\mu_2)$. Moreover, we have the following theorem. Since the proof of the theorem is similar to that of the scalar-valued case which is available in \cite[Theorem 3.12]{MRV}, we omit it. 
\begin{theorem}
    For any $f \in \mathcal{O}(\mathbb{D}^2, \mathcal{E}),$ the following statements are equivalent.
    \begin{enumerate}
        \item $f \in \mathcal{D}^2_{\mathcal{E}}(\mu_1,\mu_2)$
        \item $z_1f \in \mathcal{D}^2_{\mathcal{E}}(\mu_1,\mu_2)$
        \item $z_2f \in \mathcal{D}^2_{\mathcal{E}}(\mu_1,\mu_2). $
    \end{enumerate}
\end{theorem} 
    We also list out some of the properties of the vector-valued Dirichlet-type spaces in the following theorem. We do not provide the proofs since it is verbatim with the scalar-valued case. For more details we refer the reader to \cite{MRV}. 
    \begin{theorem}
    Suppose $\mu_1$ and $\mu_2$ are two positive $\mathcal B(\mathcal E)$-valued operator measure with $\mathcal{E}$ being a Hilbert space. Then  
    \begin{itemize}
        \item[(a)] the set of all polynomials in two variables with coefficients from $\mathcal{E}$ is a dense set in $\mathcal{D}^2_{\mathcal{E}}(\mu_1,\mu_2)$, 
        \item[(b)] the multiplication operators $M_{z_1},M_{z_2}$  are analytic $2$-isometries on $\mathcal{D}^2_{\mathcal{E}}(\mu_1,\mu_2)$. 
    \end{itemize}    
    \end{theorem}

\section{von Neumann Wold-type decomposition for Analytic pair}
    In this section, we deal with a pair of doubly commuting analytic $2$-isometries on a Hilbert space and completely characterize them in terms of shifts on Dirichlet-type spaces over bidisc. 
    
    We begin with a pair of commuting analytic $2$-isometries $(T_1, T_2)$ defined on a Hilbert space $\mathcal{H}$. Using the definition of $2$-isometry, it is straightforward to show that $T_1$ and $T_2$ are left-invertible operators and we consider the left-inverses $L_i:=(T_i^*T_i)^{-1}T_i^*$ for $i=1,2.$ Also, note that for each $i=1,2$, the operator $P_i:= I - T_iL_i$ is the orthogonal projection of $\mathcal{H}$ onto $\mathcal{E}_i :=\mathcal{H}\ominus T_i(\mathcal{H})=\ker T_i^*$. The defect operator corresponding to $T_i$ is $D_i:=(T_i^*T_i-I)^{1/2}$, the corresponding defect space is  $\mathcal{D}_i:=\overline{D_i(\mathcal{H})}$ and $P_{\mathcal{D}_i}$ is the orthogonal projection onto $\mathcal{D}_i$. For this pair $(T_1,T_2)$ on $\mathcal{H}$, the $2$-isometric property of each $T_i$ on $\mathcal{H}$ induces the map $\Tilde{T_i}:\mathcal{D}_i \rightarrow \mathcal{D}_i$, defined by $D_ix \mapsto D_iT_ix$, which turns out to be an isometry and therefore it has a unitary extension on some Hilbert space $\mathcal{K}$. 
    Moreover, using spectral theorem,  for each integer $m \geq 0$, we have 
     $$ 
     \tilde{T_i}^m = \int_{\mathbb{T}} e^{im\theta} dE_i(e^{i\theta}),
     $$ where $E_i:\Omega \rightarrow \mathcal{B}(\mathcal{K})$ denotes the spectral measure corresponding to the unitary extension of $\Tilde{T_i}$. A direct application of Theorem \ref{olofsson-analytic} (see also \cite[Theorem 4.1]{Olofsson}) gives the existence of positive $\mathcal{B}(\mathcal{E}_i)$-valued operator measures $\mu_i$ defined on $\mathbb{T}$ by      \begin{eqnarray}\label{explicit - mu-i}
         \mu_i(\sigma) = P_iD_iP_{\mathcal{D}_i}E_i(\sigma)D_i \vert_{\mathcal{E}_i}
     \end{eqnarray} 
     such that $T_i$ is unitarily equivalent to $M_z$ on $\mathcal D_{\mathcal E_i}(\mu_i).$ 
     For more details, we refer \cite{Olofsson} to the reader.

     If the pair $(T_1,T_2)$ is doubly commuting, that is, $T_1T_2 = T_2T_1$ and $T_1T_2^* = T_2^*T_1$ then 
     \[ 
     (I - T_1L_1)(I-T_2L_2) = (I-T_2L_2)(I-T_1L_1),
     \] 
     that is, the orthogonal projections $P_1=I - T_1L_1$ and $P_2=I - T_2L_2$ are commutative. Hence, the product $P:=P_1P_2$ is also an orthogonal projection onto $\mathcal E:= \ker T_1^* \cap \ker T_2^*$  and due to \cite[Lemma 4.1]{MRV}, the subspace $\mathcal{E}$ is a non-trivial subspace of $\mathcal{H}$. We consider for $i=1,2$, $\Tilde{\mu_i}(\sigma) := PD_iP_{\mathcal{D}_i}E_i(\sigma)D_i \vert_{\mathcal{E}}$ which turns out to be positive $\mathcal{B}(\mathcal{E})$-valued operator measures defined on $\mathbb{T}$.

    In the following lemma, we recall an important operator identity for analytic $2$-isometry, proved in \cite{Richter2}. 
    \begin{lemma}
    Let $S \in \mathcal{B}(\mathcal{H})$ be an analytic $2$-isometry. Then the operators $L_S := (S^*S)^{-1}S^*$, $P_S= I - SL_S$ and $D_S = (S^*S - I)^{\frac{1}{2}}$ satisfy
\begin{equation}\label{norm identity}
    \|x\|^2=\sum_{k=0}^\infty \|P_SL_S^kx\|^2 + \sum_{k=1}^\infty \|D_SL_S^kx\|^2 \quad \quad \quad ~~  \text{for all} ~~ x \in \mathcal{H}.
\end{equation}
\end{lemma}     
    The map $V$ in the following proposition is motivated by \cite{Olofsson} and serves as an isometry from $\mathcal{H}$ onto $\mathcal{D}^2_{\mathcal{E}}(\Tilde {\mu_1}, \Tilde{\mu_2}).$ This proposition plays a crucial role in this article.
\begin{prop}
    Let $(T_1, T_2)$ be a pair of analytic doubly commuting $2$-isometries on a Hilbert space $\mathcal{H}$ 
    Then the maps $\Tilde{\mu_1}$ and $\Tilde{\mu_2}$, defined by $\Tilde{\mu_i}(\sigma)=PD_iP_{\mathcal{D}_i}E_i(\sigma)D_iP$ for $i=1,2$, are positive $\mathcal{B}(\mathcal{E})$-valued operator measures on $\mathbb{T}$ and the map $V:x \mapsto Vx$ defined by
\begin{equation}\label{mapping V}
     Vx(z_1,z_2) =\sum_{m = 0}^{\infty} \sum_{n = 0}^{\infty} (PL_1^mL_2^n x) z_1^m z_2^n 
   \end{equation} 
   is an isometry from $\mathcal{H}$ onto $\mathcal{D}^2_{\mathcal{E}}(\Tilde{\mu_1},\Tilde{\mu_2})$. 
\end{prop}
\begin{proof}
    By hypothesis for $i=1,2$, each $T_i$ is analytic $2$-isometry on $\mathcal{H}$, and therefore by applying Theorem \ref{olofsson-analytic}, we get positive $\mathcal{B}(\mathcal{E}_i)$-valued operator measures say $\mu_1,\mu_2$ and unitary maps $V_i$ from $\mathcal{H}$ onto $\mathcal{D}(\mu_i)$ defined by $(V_ix)(z) = \sum_{m \geq 0} (P_iL_i^mx) z^m$ such that $V_iT_i=M_{z_i}V_i$ for $i=1,2$.  
    From \eqref{norm identity}, for each $i=1,2$ and $x \in \mathcal{H},$ we have 
\begin{equation}\label{norm identity2}
    \|x\|^2 = \sum_{k=0}^{\infty} \|P_iL_i^kx\|^2 +  \sum_{k=1}^\infty \|D_iL_i^kx\|^2.
\end{equation} Consider the equation (\ref{norm identity2}) for $i=1$, replace $x$ by $P_2L_2^nx$ and then by taking sum over $n \in \mathbb{Z}_{\geq 0},$ we get 
\begin{align} \label{eqn1-RHS norm x}   \sum_{n=0}^\infty\|P_2L_2^nx\|^2 &= \sum_{m=0}^\infty \sum_{n=0}^\infty \|PL_1^mL_2^nx\|^2 +  \sum_{m=1}^\infty \sum_{n=0}^\infty \|D_1L_1^mP_2L_2^nx\|^2.  
\end{align}    
    Similarly $x$ is replaced with $D_2L_2^nx$ in the equation \eqref{norm identity2} for $i=1$ and then taking sum over $n \in \mathbb{N},$ we get 
\begin{align} \label{eqn2-RHS norm x}   
\sum_{n=1}^\infty\|D_2L_2^nx\|^2 &= \sum_{m=0}^\infty \sum_{n=1}^\infty \|P_1L_1^mD_2L_2^nx\|^2 +  \sum_{m=1}^\infty \sum_{n=1}^\infty \|D_1L_1^mD_2L_2^nx\|^2.  
\end{align}
    Let $D$ denote the operator $D_1D_2$ and $L^{m,n}$ represent the operator $L_1^mL_2^n.$ 
    Using \eqref{eqn1-RHS norm x} and \eqref{eqn2-RHS norm x} in \eqref{norm identity2}, we obtain
\begin{equation}\label{two variable norm identity}
       \|x\|^2 = \sum_{\substack{m=0 \\ n=0}}^\infty  \|PL^{m,n}x\|^2 +  \sum_{\substack{m=1 \\  n=0}}^\infty  \|D_1P_2L^{m,n}x\|^2   +  \sum_{\substack{m=0 \\ n=1}}^\infty  \|P_1D_2L^{m,n}x\|^2 + \sum_{\substack{m=1 \\ n=1}}^\infty \|DL^{m,n}x\|^2.  
\end{equation}
    Moreover, by hypothesis the pair $(T_1,T_2)$ is doubly commuting and hence, by using  \cite[Theorem 4.3]{MRV}, the pair has the wandering subspace property. Therefore, the given Hilbert space can be decomposed as 
    \[
    \mathcal{H}=\bigvee_{m,n \geq 0}T_1^mT_2^n(\mathcal{E}).
    \] 
    Now, with the help of two positive measures $\mu_1$ on $\mathcal{B}(\mathcal{E}_1)$ and $\mu_2$ on $\mathcal{B}(\mathcal{E}_2)$ given in \eqref{explicit - mu-i}, we define two $\mathcal{B}(\mathcal{E})$-valued measures $ \Tilde{\mu_1}, \Tilde{\mu_2}$ on $\mathbb{T}$ as
    \[ 
   \Tilde{\mu_i}(\sigma)=PD_iP_{\mathcal{D}_i}E_i(\sigma)D_iP, \quad \quad i=1,2.
    \] 
    It is easy to deduce that $\Tilde{\mu_1}(\sigma) = P_2 \mu_1(\sigma) P_2$ and $\Tilde{\mu_2}(\sigma) = P_1 \mu_2(\sigma) P_1$. Thus both the measures $ \Tilde{\mu_1}$ and $\Tilde{\mu_2}$ are positive. We consider the Dirichlet-type space $\mathcal{D}^2_{\mathcal{E}}(\Tilde{\mu_1},\Tilde{\mu_2})$ corresponding to the measures $\Tilde{\mu_1}, \Tilde{\mu_2}$ and as well as the map defined in \eqref{mapping V}. 
    Then for any $x$ in $\mathcal{H}$ of the form  $\sum_{m,n=0}^k T_1^{m}T_2^{n}a_{m,n} \in \mathcal{H}$ with $a_{m,n} \in \mathcal{E}$ and for all $(z_1,z_2) \in \mathbb{D}^2$ we have 
    \[
     (Vx)(z_1,z_2) = \sum_{m,n =  0 }^k \sum_{k_1,k_2 = 0}^{\infty}PL_1^{k_1}L_2^{k_2} T_1^m T_2^n a_{m,n} z_1^{k_1} z_2^{k_2} 
    = \sum_{m,n = 0}^k a_{m,n} z_1^m z_2^n.
    \] 
    Note that, $\|Vx\|^2_{H^2(\mathbb{D}^2)}=\sum_{m,n=0}^{\infty} \|PL^{m,n}x\|^2$ for each $x \in \mathcal{H}$. Also, using the doubly commuting property of the pair $(T_1,T_2)$ on $\mathcal{H}$ and representation of isometry $\Tilde{T_1}$, we derive for each $x \in \mathcal{H}$ and $n_1,n_2 \in \mathbb{Z}_{\geq 0},$
    

\begin{align*}
    D_1P_2L^{n_1,n_2}x &= \sum_{m,n = 0}^{\infty}  D_1L_1^{n_1}P_2L_2^{n_2} T_1^m T_2^n x_{m,n} \\
    &= \sum_{m = n_1}^{\infty} D_1 T_1^{m-n_1} x_{m,n_2} \\
    &= \sum_{m = 0}^{\infty} \int_{\mathbb{T}} e^{im\theta_1} dE_1(e^{i\theta_1}) D_1x_{m+n_1, n_2}.   
\end{align*} 
    Now using the above identity we compute the norm 
\begin{align*}
    \|D_1P_2L^{n_1,n_2}x\|^2 
    &= \sum_{j_1,k_1 = 0}^{\infty} \langle P_1D_1P_{\mathcal{D}_1}\int_{\mathbb{T}} e^{i(j_1-k_1)\theta_1} dE_1(e^{i\theta_1})D_1x_{j_1+n_1,n_2}, x_{k_1+n_1,n_2} \rangle \\
    &=  \sum_{j_1,k_1 = n_1}^{\infty} \langle  \widehat{\mu_1}(k_1-j_1) x_{j_1,n_2}, x_{k_1,n_2} \rangle
\end{align*}
Combining all the above identities we derive 
\begin{align}\label{norm identity_2}
    \sum_{n_2 = 0}^{\infty} \sum_{n_1 = 1}^{\infty} \|D_1P_2L^{n_1,n_2}x\|^2 
    &= \sum_{n_2 = 0}^{\infty} \sum_{j_1,k_1 = 1}^{\infty} (j_1 \wedge k_1)  \langle  \widehat{\mu_1}(k_1-j_1) x_{j_1,n_2}, x_{k_1,n_2} \rangle.
\end{align}
Similarly, we can derive
\begin{equation} \label{norm identity_3}
    \sum_{n_1=0}^\infty \sum_{n_2=1}^\infty \|D_2P_1L^{n_1,n_2}x\|^2 =   \sum_{n_1 = 0}^{\infty} \sum_{j_2,k_2 = 1}^{\infty} (j_2 \wedge k_2) \langle \hat{\mu_2}(k_2-j_2)x_{n_1,j_2}, x_{n_1,k_2} \rangle.
\end{equation}
Also, note that following a similar calculation we get 
\begin{align*}
   DL^{n_1,n_2}x &=  \sum_{m,n=0}^{\infty}  D_1L_1^{n_1}D_2L_2^{n_2} T_1^m T_2^n x_{m,n} \\
   &= \sum_{m,n = 0}^{\infty} \bigg( \int_{\mathbb{T}} e^{im\theta_1} dE_1(e^{i\theta_1}) \bigg)D_1 \bigg(\int_{\mathbb{T}} e^{in\theta_2} dE_2(e^{i\theta_2})\bigg) D_2 x_{m+n_1, n+n_2}. 
\end{align*}
And using the above identity we calculate the norm
\begin{align*}
  \| DL^{n_1,n_2}x \|^2  
  &= \sum_{j_1,k_1 = n_1}^{\infty} \sum_{j_2,k_2= n_2}^{\infty} \langle \widehat{\mu_1} (k_1-j_1)  \widehat{\mu_2} (k_2-j_2) x_{j_1,j_2}, x_{k_1,k_2} \rangle.
\end{align*}
Therefore, we have 
\begin{align}\label{norm identity_4}
   \sum_{n_1, n_2 = 1}^{\infty}  \| DL^{n_1,n_2}x \|^2  
   &=  \sum_{j_1,k_1 = 1}^{\infty} \sum_{j_2,k_2 = 1}^{\infty} (j_1 \wedge k_1) ( j_2 \wedge k_2) \langle \widehat{\mu_1} (k_1-j_1)  \widehat{\mu_2} (k_2-j_2) x_{j_1,j_2}, x_{k_1,k_2} \rangle. 
\end{align} 
    Following the description of the term $D^{(2)}_{\Tilde{\mu}_1,\Tilde{\mu}_2,1}(Vx)$ in \eqref{dirichlet integral_1} and comparing it with the equation \eqref{norm identity_2} we deduce for any $r \in (0,1)$
    \begin{align*}
     D^{(2)}_{\Tilde{\mu}_1,\Tilde{\mu}_2,1}(Vx) &= \lim_{r \rightarrow 1^-}  \sum_{n=0}^\infty \sum_{m,p=1}^\infty (m \wedge p)  \langle  \widehat{\mu_1}(p-m) x_{m,n}, x_{p,n} \rangle r^{2n} \\
     &= \lim_{r \rightarrow 1^-}  \sum_{n_2=0}^{\infty} \sum_{n_1 = 1}^{\infty} \|D_1 P_2L^{n_1,n_2}x\|^2 r^{2n_2}
\end{align*}
and note that for each fixed $r \in (0,1)$, 
\begin{align*}
    \sum_{n_1 = 1}^{\infty} \|D_1 P_2L^{n_1,n_2}x\|^2 r^{2n_2} <  \sum_{n_1 = 1}^{\infty} \|D_1 P_2L^{n_1,n_2}x\|^2.
\end{align*}
Thus with the help of Dominated convergence theorem we can interchange the limit and the sum. Therefore it shows that 
\begin{align}\label{1}
     D^{(2)}_{\Tilde{\mu}_1,\Tilde{\mu}_2,1}(Vx) &=   \sum_{n_2 = 0}^{\infty} \sum_{n_1 = 1}^{\infty} \|D_1 P_2L^{n_1,n_2}x\|^2.
\end{align}
Similarly, we obtain the following identities by comparing the equations \eqref{dirichlet integral_2} with \eqref{norm identity_3} and \eqref{dirichlet integral_3} with \eqref{norm identity_4}
\begin{align}
    D^{(2)}_{\Tilde{\mu}_1,\Tilde{\mu}_2,2}(Vx) &=   \sum_{n_1=0}^{\infty} \sum_{n_2 = 1}^{\infty} \|P_1 D_2L^{n_1,n_2}x\|^2~~~ \mbox{and} \label{eq:second_eq1} \\
     D^{(2)}_{\Tilde{\mu}_1,\Tilde{\mu}_2,3}(Vx) &=   \sum_{n_1=1}^{\infty} \sum_{n_2 = 1}^{\infty} \|D L^{n_1,n_2}x\|^2.  \label{eq:second_eq2}   
\end{align}
 The equations  \eqref{1}, \eqref{eq:second_eq1}, \eqref{eq:second_eq2} put together in \eqref{two variable norm identity} establishes the fact that $V$ is an isometry.



    Since, by the Lemma \cite[Lemma 3.14] {MRV}, polynomials are dense in the space $\mathcal{D}^2_{\mathcal{E}}(\Tilde{\mu_1},\Tilde{\mu_2})$, therefore it is enough to show that for any polynomial $p \in \mathcal{D}^2_{\mathcal{E}}(\mu_1,\mu_2)$, with the expression $p(z_1,z_2)=\sum_{m,n = 0}^k a_{m,n} z_1^m z_2^n$, where $a_{m,n} \in \mathcal{E}$ there exists an element $x \in \mathcal{H}$ such that $Vx=p$. By using the hypothesis of this proposition we get for any $(z_1,z_2) \in \mathbb{D}^2$
    \[
    V \Big(\sum_{m,n=0}^k T_1^{m} T_2^{n}a_{m,n} \Big)(z_1,z_2) = \sum_{m,n =  0 }^k \sum_{k_1,k_2 = 0}^{\infty} PL_1^{k_1}L_2^{k_2} T_1^m T_2^n a_{m,n} z_1^{k_1} z_2^{k_2} 
    = \sum_{m,n = 0}^k a_{m,n} z_1^m z_2^n.
    \]
    Therefore, $x = \sum_{m,n=0}^k T_1^{m} T_2^{n}a_{m,n}$ satisfies the equation $Vx=p$ and since the map $V$ is an isometry therefore the value of $x \in \mathcal{H}$ is also unique. This completes the proof.
\end{proof}

The following theorem gives a complete characterization of the class of pair of doubly commuting analytic $2$-isometries acting on a Hilbert space.
\begin{theorem}\label{model thm with analyticity}
   Let $(T_1, T_2)$ be a pair of analytic doubly commuting $2$-isometries on a Hilbert space $\mathcal{H}$. Then $(T_1,T_2)$ on $\mathcal{H}$ is unitarily equivalent with $(M_{z_1},M_{z_2})$ on $\mathcal{D}^2_{\mathcal{E}}(\Tilde{\mu_1},\Tilde{\mu_2}),$ where $\Tilde{\mu_i}(\sigma)=PD_iP_{\mathcal{D}_i}E_i(\sigma)D_iP$ for $i=1,2$, are positive $\mathcal{B}(\mathcal{E})$-valued operator measures. 
\end{theorem} 
\begin{proof}
    In the previous proposition, we proved that the map $V: \mathcal{H} \to \mathcal{D}^2_{\mathcal{E}}(\Tilde{\mu_1},\Tilde{\mu_2}),$ defined above, is a unitary. Therefore it is enough to prove that   
    $VT_i = M_{z_i}V,$ for $i=1,2.$ 
    For $z_1,z_2 \in \mathbb{D}$, we obtain
\begin{align*}
    (VT_1x)(z_1,z_2) &= \sum_{m = 0}^{\infty} \sum_{n = 0}^{\infty} (PL_1^mL_2^nT_1x) z_1^m z_2^n 
    = \sum_{m = 1}^{\infty} \sum_{n = 0}^{\infty} (PL_1^{m-1}L_2^nx) z_1^m z_2^n \\
    &= z_1 \sum_{m = 1}^{\infty} \sum_{n = 0}^{\infty} (PL_1^{m-1}L_2^nx) z_1^{m-1} z_2^n \\
    &= M_{z_1}(Vx)(z_1,z_2).
\end{align*}
    Similarly, it can be verified that $VT_2=M_{z_2}V.$
    This guarantees that $(T_1,T_2)$ on $\mathcal{H}$ is unitarily equivalent with $(M_{z_1},M_{z_2})$ on $\mathcal{D}^2_{\mathcal{E}}(\Tilde{\mu_1},\Tilde{\mu_2}).$
\end{proof}
\begin{rem}\label{uniqueness of measure upto a unitary}
    As is shown in \cite[Remark 5.2]{MRV}, it can be shown that if the pair $(M_{z_1},M_{z_2})$ on $\mathcal{D}^2_{\mathcal{E_{\eta}}}(\eta_1,\eta_2)$ 
    is unitarily equivalent to the pair $(M_{z_1},M_{z_2})$ 
    on $\mathcal{D}^2_{\mathcal{E_\beta}}(\beta_1,\beta_2)$ 
    for some finite positive operator-valued Borel measures 
    $\eta_1,\eta_2,\beta_1$ and $\beta_2$ on $\mathbb T$ 
    with $\mathcal{E_\eta}$ and $\mathcal{E_\beta}$ are joint kernel of $(M_{z_1}^*, M_{z_2}^*)$ on $\mathcal{D}^2_{\mathcal{E_{\eta}}}(\eta_1,\eta_2)$ and $\mathcal{D}^2_{\mathcal{E_\beta}}(\beta_1,\beta_2)$ respectively, then there exists a unitary operator $U:\mathcal E_{\eta
    } \to \mathcal E_{\beta}$ such that $\eta_i(\sigma)=U^*\beta_i(\sigma) U$ for all Borel subsets $\sigma$ of unit circle $\mathbb T.$ 
\end{rem}
\section{von Neumann Wold type decomposition for General pair}
    The main purpose of this section is to develop a von Neumann Wold-type decomposition for the class of doubly commuting $2$-isometries $(T_1,T_2)$ acting on a Hilbert space $\mathcal{H}$. The key idea is to use the Wold-type decomposition for a single $2$-isometry which we recall first for the convenience of the reader.  
    Given a $2$-isometry $T$ on a  Hilbert space $\mathcal{H}$, in \cite{Shimorin, Olofsson_2} the authors show that the operator $T$ can be decomposed in a direct sum of unitary operator and an analytic $2$-isometry.
  More precisely,  
    \begin{theorem}\label{2-isometry thm}
    Let $T \in \mathcal{B}(\mathcal{H})$ be a 2-isometry. Then the space $\mathcal{H}$ admits an unique decomposition of the form $ \mathcal{H}=\mathcal{H}_0 \oplus \mathcal{H}_1$, where the subspaces $\mathcal{H}_0$ and $\mathcal{H}_1$ are $T$-reducing subspaces of $\mathcal{H}$ and $T|_{\mathcal{H}_0}$ is unitary. 
    Moreover, 
    \[ 
    \mathcal{H}_{0} = \bigcap_{n \geq 0}T^n(\mathcal{H}) \quad \text{and} \quad \mathcal{H}_1 = \bigvee_{n \geq 0} T^n(\mathcal{E}_T)
    \] where $\mathcal{E}_T=\mathcal{H}\ominus T(\mathcal{H})$ is the wandering subspace for $T$.  
    \end{theorem}    
Before proceeding further, we note down some important properties for the pair of doubly commuting $2$-isometry operators. We skip the proof as these are straightforward properties to verify.
\begin{lemma}\label{reducing}
   Let $(T_1,T_2)$ be a pair of doubly commuting $2$-isometries on $\mathcal{H}$. Then for $i,j \in \{1,2\}$ with $i \neq j$, 
   \begin{enumerate}
       \item[(a)] $kerT_i^*$ is $T_j$-reducing subspace of $\mathcal{H}$.
       \item[(b)] For each $m \geq 0$, $T_i^m\bigg ( \bigcap_{n \geq 0} T_j^{n}\mathcal{E}_i\bigg) = \bigg( T_i^m (\mathcal{E}_i)\bigg) \cap \bigg(\bigcap_{n \geq 0} T_j^{n} \mathcal{H} \bigg)$. 
   \end{enumerate} 
\end{lemma}
         In the following theorem, we get an analogue of Theorem \ref{2-isometry thm}. 
         More precisely, we present the decomposition of the Hilbert space $\mathcal{H}$ corresponding to a pair of doubly commuting $2$-isometries acting on a Hilbert space $\mathcal{H}.$ The presence of analytic structure in this result is certainly giving an improved version of Theorem 5.1 in \cite{NZ} in this set-up. 
    

\begin{theorem}\label{wold-type decomposition}
     Let $(T_1,T_2)$ be a pair of doubly commuting $2$-isometries on a Hilbert space $\mathcal{H}$. Then there exist positive Borel measures $\nu_1,\nu_2,\eta_1$ and $\eta_2$ on $\mathbb{T}$ such that 
     \[
     \mathcal{H} \cong \mathcal{H}_{00} \oplus \mathcal{D}_{\mathcal{E}_{10}}(\nu_1) \oplus \mathcal{D}_{\mathcal{E}_{01}}(\nu_2) \oplus \mathcal{D}^2_{\mathcal{E}}(\eta_1,\eta_2)
     \] where $\mathcal{E}_{10} = \cap_{n \geq 0} T_2^n \, kerT_1^*$, $\mathcal{E}_{01} = \cap_{m \geq 0} T_1^m \, kerT_2^*$ and $\mathcal{E} = kerT_1^* \cap kerT_2^*$.  Moreover, 
                  \[ T_1 \cong \begin{bmatrix}
U_0 & 0 & 0 & 0\\
0 & M_z & 0 & 0 \\
0 & 0 & U_1 & 0 \\
0 & 0 &  0  &  M_{z_1}
\end{bmatrix}    \quad \text{and} \quad   T_2 \cong \begin{bmatrix}
V_0 & 0 & 0 & 0\\
0 & V_1 & 0 & 0 \\
0 & 0 & M_z & 0 \\
0 & 0 &  0  &  M_{z_2}
\end{bmatrix}
\]  where $U_0$ and     $V_0$ are unitary   on $\mathcal{H}_{00}    = \cap_{m,n \geq 0} T_1^mT_2^n \mathcal{H}$, $U_1$ and $V_1$ are unitary on $\mathcal{D}_{\mathcal{E}_{01}}(\nu_2)$ and  $\mathcal{D}_{\mathcal{E}_{10}}(\nu_1)$, respectively. 
    \end{theorem}
    
\begin{proof}
    By hypothesis, $T_1$ is a $2$-isometry on $\mathcal{H}$. Then, by applying Theorem \ref{2-isometry thm}, the Hilbert space $\mathcal{H}$ can be decomposed into two $T_1$-reducing subspace, that is 
\begin{align*}
     \mathcal{H} &= \mathcal{H}_0 \oplus \mathcal{H}_1 
\end{align*}
    where $\mathcal{H}_0 =\bigcap_{m \geq 0}T_1^m(\mathcal{H})$ and $\mathcal{H}_1 = \bigvee_{m \geq 0}T_1^m(\mathcal{E}_1)$ with $\mathcal{E}_1 = kerT_1^*$. Since the pair $(T_1,T_2)$ on $\mathcal{H}$ is doubly commutative, so by the Lemma \ref{reducing}, both the subspaces $\mathcal{H}_0$ and $\mathcal{H}_1$ are $T_2$-reducing. Thus application of the Theorem \ref{2-isometry thm} for the operator $T_2|_{\mathcal{H}_0}$ on $\mathcal{H}_0$ yields the decomposition of $\mathcal{H}_0$ into two $T_2$-reducing subspaces. That is, 
\begin{align*}
   \mathcal{H}_0 &= \bigcap_{n \geq 0}T_2^n(\mathcal{H}_0) \oplus \bigvee_{n \geq 0} T_2^n(\mathcal{H}_0 \ominus T_2\mathcal{H}_0) \nonumber \\
   &= \bigcap_{m,n \geq 0} T_1^m T_2^n \mathcal{H} \oplus \bigvee_{n \geq 0} T_2^n \bigg(\bigcap_{m \geq 0}T_1^m(\mathcal{H} \ominus T_2\mathcal{H})\bigg),
\end{align*}
    where the last equality follows by using the doubly commutativity property of the pair $(T_1,T_2)$. We denote $\mathcal{H}_{00} = \bigcap_{m,n \geq 0} T_1^m T_2^n \mathcal{H}$ and $\mathcal{H}_{01}= \bigvee_{n \geq 0} T_2^n \big(\bigcap_{m \geq 0}T_1^m\mathcal{E}_2  \big)$. Note that, for $i=1,2$, following Theorem \ref{2-isometry thm}, the subspace $\cap_{m \geq 0}T_i^m \mathcal{H}$ is $T_i$-reducing and $T_i|_{\cap_{m \geq 0}T_i^m \mathcal{H}}$ is unitary. Since $\mathcal{H}_{00} \subseteq \cap_{m \geq 0}T_i^m \mathcal{H}$ is a $T_i$-reducing subspace for any $i=1,2$, hence $T_1|_{\mathcal{H}_{00}}$ and $T_2|_{\mathcal{H}_{00}}$ both are unitary on $\mathcal{H}_{00}$. Also, by Lemma \ref{reducing}, we know that the space 
    \[ 
    \mathcal{H}_{01}= \big(\bigvee_{n \geq 0} T_2^n \mathcal{E}_2 \big) \cap \big(\bigcap_{m \geq 0}T_1^m \mathcal{H} \big). 
    \] That is, $\mathcal{H}_{01}$ is the intersection of two joint $(T_1,T_2)$-reducing subspace of $\mathcal{H}$, therefore it is also a joint $(T_1,T_2)$-reducing subspace. Along with that, since $\mathcal{H}_{01} \subseteq \bigcap_{m \geq 0}T_1^m \mathcal{H} = \mathcal{H}_0$ and $T_1|_{\mathcal{H}_0}$ is a unitary, so $T_1|_{\mathcal{H}_{01}}$ is also an unitary. 
    

     Now we apply Theorem \ref{2-isometry thm} for the operator $T_2|_{\mathcal{H}_1}$ on $\mathcal{H}_1$. That yields the decomposition of $\mathcal{H}_1$ into two $T_2$-reducing subspaces. That is, by using the hypothesis on the pair we have
\begin{align*}
     \mathcal{H}_1 &= \bigcap_{n \geq 0} T_2^n \big(\mathcal{H}_1\big) \oplus \bigvee_{n \geq 0} T_2^n( \mathcal{H}_1 \ominus T_2\mathcal{H}_1) \\
     &=  \bigvee_{m \geq 0} T_1^m\bigg ( \bigcap_{n \geq 0} T_2^{n}\mathcal{E}_1\bigg) \nonumber \oplus \bigvee_{m,n \geq 0} T_1^m T_2^n( \mathcal{E}_1 \cap  \mathcal{E}_2).
\end{align*}
    We denote $\mathcal{H}_{10} = \bigvee_{m \geq 0} T_1^m \big( \bigcap_{n \geq 0} T_2^{n}\mathcal{E}_1 \big)$ and $\mathcal{H}_{11} = \bigvee_{m,n \geq 0} T_1^m T_2^n( \mathcal{E}_1 \cap  \mathcal{E}_2)$. Again by using the Lemma \ref{reducing} we have $\mathcal{H}_{10} = \big(\bigvee_{m \geq 0} T_1^m \mathcal{E}_1 \big) \cap \big(\bigcap_{n \geq 0} T_2^{n} \mathcal{H} \big)$. Being intersection of two joint $(T_1,T_2)$-reducing subspace, the space $\mathcal{H}_{10}$ is also a joint $(T_1,T_2)$-reducing subspace of $\mathcal{H}$. Since $\mathcal{H}_{10} \subseteq \bigcap_{n \geq 0} T_2^{n} \mathcal{H}$, so $T_2|_{\mathcal{H}_{10}}$ is an unitary. 
    Since $\mathcal{H}_{11}= \mathcal{H}_1 \ominus \mathcal{H}_{10}$, so $\mathcal{H}_{11}$ is also a joint $(T_1,T_2)$-reducing subspace. Moreover, the pair $(T_1|_{\mathcal{H}_{11}}, T_2|_{\mathcal{H}_{11}})$ is a doubly commuting pair of analytic $2$-isometries on $\mathcal{H}_{11}$. 
    Combining all the above decompositions we have 
    \[
    \mathcal{H} = \mathcal{H}_{00}  \oplus \mathcal{H}_{10}  \oplus \mathcal{H}_{01}  \oplus \mathcal{H}_{11}   
    \] and for $\alpha_1, \alpha_2 \in \{0,1\}$, $T_i$ on $\mathcal{H}_{\alpha_1,\alpha_2}$ is unitary if $\alpha_i=0$ for each $i=1,2.$ 
    From the above decomposition of the Hilbert space $\mathcal{H}$, we know $T_1|_{\mathcal{H}_{10}}$ is also a $2$-isometry on $\mathcal{H}_{10}$. And, by Lemma \ref{reducing}(b), we note that 
    \[
    T_1^m\bigg ( \bigcap_{n \geq 0} T_2^{n}\mathcal{E}_1\bigg) = \bigg( T_1^m (\mathcal{E}_1)\bigg) \cap \bigg(\bigcap_{n \geq 0} T_2^{n} \mathcal{H} \bigg) 
    \] for each $m \geq 0$. Now, using the fact that $\mathcal{E}_1$ is a wandering subspace for $T_1$ on $\mathcal{H}$, we conclude that $\cap_{n \geq 0} T_2^{n}\mathcal{E}_1$ is a wandering subspace for $T_1|_{\mathcal{H}_{10}}$ on $\mathcal{H}_{10}$ and the description of the subspace $\mathcal{H}_{10}$ helps to conclude that $T_1|_{\mathcal{H}_{10}}$ is an analytic $2$-isometric operator on $\mathcal{H}_{10}$. Hence, by applying Theorem \ref{olofsson-analytic}, we can immediately conclude that 
    \[ 
    (T_1|_{\mathcal{H}_{10}}, \mathcal{H}_{10}) \cong (M_z, \mathcal{D}_{\mathcal{E}_{10}}(\nu_1))
    \] for some positive Borel measure $\nu_1$ on $\mathbb{T}$ with $\mathcal{E}_{10} = \cap_{n \geq 0} T_2^n\mathcal{E}_1$. Analogously, we can also prove that the subspace $\mathcal{E}_{01} = \cap_{m \geq 0} T_1^m \, kerT_2^*$ is a wandering subspace for $T_2|_{\mathcal{H}_{01}}$ and $(T_2|_{\mathcal{H}_{01}}, \mathcal{H}_{01}) \cong (M_z, \mathcal{D}_{\mathcal{E}_{01}}(\nu_2))$. Moreover, from the above we know that $(T_1|_{\mathcal{H}_{11}}, T_2|_{\mathcal{H}_{11}})$ is a pair of doubly commuting analytic $2$-isometries on $\mathcal{H}_{11}$. Therefore, by intervening Theorem \ref{model thm with analyticity}, we have $(T_1|_{\mathcal{H}_{11}}, T_2|_{\mathcal{H}_{11}})$ on $\mathcal{H}_{11}$ is unitary equivalent with $(M_{z_1}, M_{z_2})$ on $\mathcal{D}^2_{\mathcal{E}}(\eta_1, \eta_2)$ for some positive measures $\eta_1$ and $\eta_2$ on $\mathbb{T}$ with $\mathcal{E} = \mathcal{E}_1 \cap \mathcal{E}_2$.
    \end{proof}

    \begin{rem}
    For a pair of doubly commuting $2$-isometries $(T_1,T_2)$ acting on a Hilbert space $\mathcal{H}$, the decomposition in Theorem \ref{wold-type decomposition} is unique in the following sense: for $\alpha_1, \alpha_2,i,j \in \{0,1\}$, if the subspace $\mathcal{K}_{\alpha_i,\alpha_j}$ is reducing for $T_i$ such that $\mathcal{H} = \mathcal{K}_{00} \oplus \mathcal{K}_{10} \oplus \mathcal{K}_{01} \oplus \mathcal{K}_{11}$ with the preperties that the restriction of  $T_i$ to $\mathcal{K}_{\alpha_i,\alpha_j}$ is unitary and analytic for $\alpha_i =0$ and $\alpha_i =1$ respectively,
    then $\mathcal{K}_{00} = \mathcal{H}_{00}$,  $\mathcal{K}_{10} = \mathcal{H}_{10}$, $\mathcal{K}_{01} = \mathcal{H}_{01}$ and $\mathcal{K}_{11} = \mathcal{H}_{11}$. Moreover the measures $\nu_1, \nu_2,\eta_1$ and $\eta_2$ in Theorem \ref{wold-type decomposition} are determined uniquely upto a unitary (see Remark \ref{uniqueness of measure upto a unitary}).
    \end{rem}

\noindent\textbf{Acknowledgements:}
The first named author's research is supported by the DST-INSPIRE Faculty Grant with Fellowship No. DST/INSPIRE/04/2020/001250 and the third named author's research work is supported by the BITS Pilani institute fellowship.

\end{document}